\documentclass[a4paper,11pt]{amsart}
\usepackage{graphicx}
 \usepackage{mathptmx}
\usepackage{amsmath}
\usepackage{amssymb}
\usepackage{enumitem}
\usepackage{xcolor}

\newmuskip\pFqmuskip

\newcommand*\pFq[6][8]{%
  \begingroup 
  \pFqmuskip=#1mu\relax
  \mathcode`=\string"8000
  \begingroup\lccode`\~=`\,
  \lowercase{\endgroup\let~}\pFqcomma
  F^{#2}_{#3}{\left(\genfrac..{0pt}{}{#4}{#5}\bigg|#6\right)}%
  \endgroup
}
\newcommand{\pFqcomma}{\mskip\pFqmuskip}

\newtheorem{theorem}{Theorem}[section]
\newtheorem{lemma}[theorem]{Lemma}
\newtheorem{corollary}[theorem]{Corollary}

\begin{document}

\title[Some identities on generalized harmonic numbers and functions]{Some identities on generalized harmonic numbers and generalized harmonic functions}

\author{Dae San  Kim }
\address{Department of Mathematics, Sogang University, Seoul 121-742, Republic of Korea}
\email{dskim@sogang.ac.kr}
\author{Hyekyung Kim}
\address{Department Of Mathematics Education, Daegu Catholic University, Gyeongsan 38430, Republic of Korea}
\email{hkkim@cu.ac.kr}
\author{Taekyun  Kim}
\address{Department of Mathematics, Kwangwoon University, Seoul 139-701, Republic of Korea}
\email{tkkim@kw.ac.kr}

\subjclass[2010]{11B68; 11B83; 26A42}
\keywords{generalized harmonic numbers of order $\alpha$; generalized harmonic function}

\maketitle

\begin{abstract}
The harmonic numbers and generalized harmonic numbers appear
frequently in many diverse areas such as combinatorial problems,
many expressions involving special functions in analytic number
theory and analysis of algorithms. The aim of this paper is to
derive some identities involving generalized harmonic numbers and
generalized harmonic functions from the beta functions
$F_{n}(x)=B(x+1,n+1),\,\,(n=0,1,2,\dots)$ using elementary
methods.
\end{abstract}

\section{Introduction}
For $s\in\mathbb{C}$, the Riemann zeta function is defined by
\begin{displaymath}
    \zeta(s)=\sum_{n=1}^{\infty}\frac{1}{n^{s}},\quad (\mathrm{Re}(s)>1),\quad (\mathrm{see}\ [1,12,14]).
\end{displaymath}
As is well known, we have
\begin{equation}
\zeta(2n)=\frac{(-1)^{n+1}B_{2n}}{2(2n)!}(2\pi)^{2n},\quad (n\in\mathbb{N}),\quad (\mathrm{see}\ [14]),\label{1}
\end{equation}
where $B_{n}$ are the Bernoulli numbers defined by
\begin{equation}
\frac{t}{e^t-1}=\sum_{n=0}^{\infty}B_{n}\frac{t^n}{n!}=1-\frac{1}{2}x+\frac{1}{6}\frac{x^2}{2!}-\frac{1}{30}\frac{x^4}{4!}+\frac{1}{42}\frac{x^6}{6!}-\frac{1}{30}\frac{x^8}{8!}+\cdots,\quad (\mathrm{see}\ [1-14]). \label{2}
\end{equation}
Thus we see from \eqref{1} and \eqref{2} that
\begin{equation}
\zeta(2)=\frac{\pi^2}{6},\quad\zeta(4)=\frac{\pi^4}{90},\quad\zeta(6)=\frac{\pi^6}{945},\quad\zeta(8)=\frac{\pi^8}{9450}, \cdots.\label{3}
\end{equation}
More generally, Hurwitz zeta function is defined by
\begin{equation}
\zeta(x,s)=\sum_{n=0}^{\infty}\frac{1}{(n+x)^{s}},\quad (\mathrm{Re}(s)>1, x >0),\quad (\mathrm{see}\ [14]).\label{4}
\end{equation} \par
The harmonic numbers are defined by
\begin{equation}
H_{0}=0,\quad H_{n}=1+\frac{1}{2}+\cdots+\frac{1}{n},\quad (n\in\mathbb{N}),\quad (\mathrm{see}\ [6-8,10,13]). \label{5}
\end{equation}
The generating function of harmonic numbers is given by
\begin{equation}
-\frac{1}{1-x}\log(1-x)=\sum_{n=1}^{\infty}H_{n}x^{n},\quad (\mathrm{see}\ [6-8,13]). \label{6}
\end{equation}
For $\alpha\in\mathbb{N}$, the generalized harmonic numbers of order $\alpha$ are defined by
\begin{equation}
H_{0}^{(\alpha)}=0,\quad H_{n}^{(\alpha)}=1+\frac{1}{2^{\alpha}}+\frac{1}{3^{\alpha}}+\cdots+\frac{1}{n^{\alpha}},\quad (\mathrm{see}\ [3-5,8]). \label{7}
\end{equation} \par
For $s\in\mathbb{C}$ with $\mathrm{Re}(s)>0$, the gamma function is given by
\begin{equation}
    \Gamma(s)=\int_{0}^{\infty}e^{-t}t^{s-1}dt,\quad (\mathrm{see}\ [9,12,14]), \label{8}
\end{equation}
with $\Gamma(s+1)=s\Gamma(s)$ and $\Gamma(1)=1$.
For $\mathrm{Re}(\alpha)>0$ and $\mathrm{Re}(\beta)>0$, the beta function is given by
\begin{equation}
B(\alpha,\beta)=\int_{0}^{1}t^{\alpha-1}(1-t)^{\beta-1}dt,\quad (\mathrm{see}\ [9,12,14]). \label{9}
\end{equation}
From \eqref{8} and \eqref{9}, we have
\begin{equation}
B(\alpha,\beta)=\frac{\Gamma(\alpha)\Gamma(\beta)}{\Gamma(\alpha+\beta)},\quad (\mathrm{see}\ [14]).\label{10}
\end{equation}\par
Lastly, the following binomial inversion theorem is well known.
\begin{theorem}
For any integer $n \ge 0$, we have
\begin{displaymath}
b_{n}=\sum_{k=0}^{n}\binom{n}{k} (-1)^{k}a_{k} \Longleftrightarrow   a_{n}=\sum_{k=0}^{n}\binom{n}{k} (-1)^{k}b_{k}.
\end{displaymath}
\end{theorem}

In this paper, we derive some expressions for $\sum_{k=0}^{n}\binom{n}{k}(-1)^{k}\frac{1}{(k+x+1)^{r}},\quad \zeta(x+1,r)$, and $\,\, r!$ in terms of generalized harmonic functions and generalized harmonic numbers using elementary methods like the binomial inversion, differentiation and integration.

\section{Some identities on generalized harmonic numbers and generalized harmonic functions}
For every $\alpha\in\mathbb{N}$, we consider the generalized
harmonic function $H_{n}(x,\alpha)$ defined by
\begin{equation}
H_{n}(x,\alpha)=\sum_{k=0}^{n}\frac{1}{(k+x+1)^{\alpha}},\quad (x\ge 0),\quad(\mathrm{see}\ [3-5]). \label{11}
\end{equation}
Note that
\begin{equation}
\frac{d}{dx}H_{n}(x,\alpha)=-\alpha H_{n}(x,\alpha+1).\label{12}
\end{equation}
In particular, for $x=0$, we have $H_{n}(0,\alpha)=H_{n+1}^{(\alpha)}$. \par
For $x >-1$ and $n=0,1,2, \dots$, we let
\begin{align}
F_n(x)&=\int_{0}^{1}(1-t)^{n}t^{x}dt=B(x+1, n+1)
=\frac{\Gamma(x+1)\Gamma(n+1)}{\Gamma(x+n+2)}\label{13}\\
&=\frac{n!}{(x+n+1)(x+n)\cdots(x+1)}=\frac{1}{\binom{x+n+1}{n}(x+1)}. \nonumber
\end{align}
By binomial theorem, \eqref{13} is also equal to
\begin{align}
F_n(x)=\int_{0}^{1}(1-t)^{n}t^{x}dt&=\sum_{k=0}^{n}\binom{n}{k}(-1)^{k}\int_{0}^{1}t^{k+x}dt \label{14} \\
&=\sum_{k=0}^{n}\binom{n}{k}(-1)^{k}\frac{1}{x+k+1}.\nonumber
\end{align}
We will use the following lemma repeatedly throughout this paper.
\begin{lemma}
Let $F_{n}(x)=B(x+1,n+1)$. Then, for any positive integer $r$, we have
\begin{align}
&(a)\quad F_{n}^{(r)}(x)=\int_{0}^{1}(1-t)^n(\log t)^{r}t^{x} dt=r!\sum_{k=0}^{n}\binom{n}{k}(-1)^{k-r}\frac{1}{(x+k+1)^{r+1}}, \nonumber \\
&(b)\quad F_{n}^{\prime}(x)=-H_{n}(x,1)F_n(x), \nonumber\\
&(c)\quad \sum_{n=1}^{\infty}\frac{1}{n}\sum_{k=0}^{n}\binom{n}{k}(-1)^{k}\frac{1}{(k+1)^{r}}=r.\nonumber
\end{align}
\end{lemma}
\begin{proof}
(a) follows from \eqref{14}. \\
(b) From \eqref{13}, we get
\begin{align*}
F_{n}^{\prime}(x)&=\frac{d}{dx}F_n(x)=\frac{d}{dx}\bigg(\frac{n!}{(x+n+1)(x+n)\cdots(x+1)}\bigg) \\
&=-\sum_{k=0}^{n}\frac{1}{k+x+1}\cdot \frac{n!}{(x+n+1)(x+n)\cdots (x+1)} \\
&=-H_{n}(x,1)F_n(x).
\end{align*}
(c) The left hand side of (c) is equal to
\begin{align*}
\sum_{n=1}^{1}&\frac{1}{n}\int_{0}^{1} \cdots \int_{0}^{1} (1-x_1 x_2 \cdots x_r)^{n} dx_1 dx_2 \cdots dx_r \\
&=-\int_{0}^{1} \cdots \int_{0}^{1} (\log x_1 + \cdots + \log x_r)\,dx_1 dx_2 \cdots dx_r =r,
\end{align*}
where the integrals are $r$-multiple ones.
\end{proof}
From Lemma 2.1, (a), (b), and \eqref{13}, we have
\begin{align}
\sum_{k=0}^{n}\binom{n}{k}(-1)^{k}\frac{1}{(k+x+1)^{2}}=H_{n}(x,1)F_n(x)
=\frac{H_{n}(x,1)}{\binom{x+n+1}{n}(x+1)}.\label{15}
\end{align}
In particular, for $x=0$, we get
\begin{equation}
\sum_{k=0}^{n}\binom{n}{k}(-1)^{k}\frac{1}{(k+1)^{2}}=\frac{1}{n+1}H_{n+1},\quad (n\ge 0),\label{16}
\end{equation}
which, by Lemma 2.1 (c), yields
\begin{equation*}
\sum_{n=1}^{\infty}\frac{1}{n(n+1)}H_{n+1}=2.
\end{equation*}
Therefore, from \eqref{15} and \eqref{16}, by binomial inversion we obtain the following theorem.
\begin{theorem}
For $n\ge 0$, we have
\begin{displaymath}
\sum_{k=0}^{n}\binom{n}{k}(-1)^{k}\frac{H_{k}(x,1)}{\binom{x+k+1}{k}(x+1)}=\frac{1}{(n+x+1)^{2}},\quad\mathrm{and}\quad \sum_{k=0}^{n}\binom{n}{k}(-1)^{k}\frac{H_{k+1}}{k+1}=\frac{1}{(n+1)^{2}}.
\end{displaymath}
\end{theorem}
We note from Theorem 2.2, \eqref{3} and \eqref{4} that
\begin{align*}
&\sum_{n=0}^{\infty}\sum_{k=0}^{n}\binom{n}{k}(-1)^{k}\frac{H_{k}(x,1)}{\binom{x+k+1}{k}(x+1)}=\sum_{n=0}^{\infty}\frac{1}{(n+x+1)^{2}}=\zeta(x+1,2),\\
&\sum_{n=0}^{\infty}\bigg(\sum_{k=0}^{n}\binom{n}{k}(-1)^{k}\frac{H_{k+1}}{k+1}\bigg)=\sum_{n=0}^{\infty}\frac{1}{(n+1)^{2}}=\frac{\pi^{2}}{6}.
\end{align*}
From Lemma 2.1 (b), \eqref{12} and \eqref{13}, we note that
\begin{align}
F_{n}^{(2)}(x)&=\frac{d^{2}}{dx^{2}}F_n(x)= (-1)\frac{d}{dx}\Big(H_{n}(x,1)F_n(x)\Big) \label{17}\\
&=(-1)^{2}H_{n}(x,2)F_n(x)+(-1)^{2}\Big(H_{n}(x,1)\Big)^{2}F_n(x), \nonumber \\
&=(-1)^{2}\Big(H_{n}(x,2)+\big(H_{n}(x,1)\big)^{2}\Big)\frac{1}{\binom{x+n+1}{n}(x+1)},\nonumber
\end{align}
and that
\begin{align}
    F_{n}^{(3)}(x)&=\frac{d^{3}}{dx^{3}}F_n(x)=\frac{d}{dx}\Big((-1)^{2}H_{n}(x,2)F_n(x)+(-1)^{2}\big(H_{n}(x,1)\big)^{2}F_n(x)\Big) \label{18}\\
    &=2(-1)^{3}H_{n}(x,3)F_n(x)+(-1)^{3}H_{n}(x,2)H_{n}(x,1)F_n(x)\nonumber \\
    &+2(-1)^{3}H_{n}(x,1)H_{n}(x,2)F_n(x)+(-1)^{3}\Big(H_{n}(x,1)\Big)^{}F_n(x)\nonumber \\
    &=2(-1)^3H_{n}(x,3)F_n(x)+3(-1)^3H_{n}(x,1)H_{n}(x,2)F_n(x)+(-1)^3\Big(H_{n}(x,1)\Big)^{3}F_n(x) \nonumber\\
    &=(-1)^3\Big(2H_{n}(x,3)+3H_{n}(x,1)H_{n}(x,2)+\big(H_{n}(x,1)\big)^{3}\Big)\frac{1}{\binom{x+n+1}{n}(x+1)}.\nonumber
\end{align}
By Lemma 2.1 (a), we get
\begin{equation}
F_{n}^{(2)}(x)=2!\sum_{k=0}^{n}\binom{n}{k}(-1)^{k-2}\frac{1}{(x+k+1)^{3}},\quad F_{n}^{(3)}(x)=3!\sum_{k=0}^{n}\binom{n}{k}(-1)^{k-3}\frac{1}{(x+k+1)^{4}}. \label{19}
\end{equation}
Therefore, by \eqref{17}-\eqref{19}, we obtain the following theorem.
\begin{theorem}
For $n\ge 0$, we have
\begin{displaymath}
\frac{H_{n}(x,2)+(H_{n}(x,1))^{2}}{\binom{x+n+1}{n}(x+1)}=2!\sum_{k=0}^{n}\binom{n}{k}(-1)^{k}\frac{1}{(x+k+1)^{3}},
\end{displaymath}
and
\begin{displaymath}
\frac{2H_{n}(x,3)+3H_{n}(x,1)H_{n}(x,2)+(H_{n}(x,1))^{3}}{\binom{x+n+1}{n}(x+1)}=3!\sum_{k=0}^{n}\binom{n}{k}(-1)^{k}\frac{1}{(x+k+1)^{4}}.
\end{displaymath}
For $x=0$, we also have
\begin{displaymath}
2!\sum_{k=0}^{n}\binom{n}{k}\frac{(-1)^{k}}{(k+1)^{3}}=\frac{H_{n+1}^{(2)}+(H_{n+1})^{2}}{n+1},
\end{displaymath}
and
\begin{displaymath}
3!\sum_{k=0}^{n}\frac{\binom{n}{k}(-1)^{k}}{(k+1)^{4}}=\frac{2H_{n+1}^{(3)}+3H_{n+1}H_{n+1}^{(2)}+(H_{n+1})^{3}}{n+1}.
\end{displaymath}
\end{theorem}
From Theorem 2.3, by binomial inversion we obtain
\begin{align}
&\frac{1}{2!}\sum_{k=0}^{n}\binom{n}{k}(-1)^{k}\frac{H_{k}(x,2)+\big(H_{k}(x,1)\big)^{2}}{\binom{x+k+1}{k}(x+1)}=\frac{1}{(x+n+1)^{3}}, \label{20} \\
&\frac{1}{3!}\sum_{k=0}^{n}\binom{n}{k}(-1)^{k}\bigg(\frac{2H_{k}(x,3)+3H_{k}(x,1)H_{k}(x,2)+(H_{k}(x,1))^{3}}{\binom{x+k+1}{k}(x+1)}\bigg)=\frac{1}{(x+n+1)^{4}}
\label{21}.
\end{align}
By \eqref{20} and \eqref{4}, we get
\begin{equation*}
    \frac{1}{2!}\sum_{n=0}^{\infty}\sum_{k=0}^{n}\frac{\binom{n}{k}(-1)^{k}}{\binom{x+k+1}{k}(x+1)}\Big(H_{k}(x,2)+\big(H_{k}(x,1)\big)^{2}\Big)=\zeta(x+1,3).
\end{equation*}
For $x=0$, we have
\begin{displaymath}
\frac{1}{2!}\sum_{n=0}^{\infty}\sum_{k=0}^{n}\frac{\binom{n}{k}(-1)^{k}}{k+1}\Big(H_{k+1}^{(2)}+(H_{k+1})^{2}\Big)=\zeta(3).
\end{displaymath}
In addition, from \eqref{21} and \eqref{4}, we also have
\begin{displaymath}
\frac{1}{3!}\sum_{n=0}^{\infty}\sum_{k=0}^{n}\frac{\binom{n}{k}(-1)^{k}}{\binom{x+k+1}{k}(x+1)}\Big(2H_{k}(x,3)+3H_{k}(x,1)H_{k}(x,2)+\big(H_{k}(x,1)\big)^{3}\Big)=\zeta(x+1,4).
\end{displaymath}
For $x=0$, from \eqref{3} we get
\begin{displaymath}
    \frac{1}{3!}\sum_{n=0}^{\infty}\sum_{k=0}^{n}\frac{\binom{n}{k}(-1)^{k}}{k+1}\Big(2H_{k+1}^{(3)}+3H_{k+1}H_{k+1}^{(2)}+(H_{k+1})^{3}\Big)=\sum_{n=0}^{\infty}\frac{1}{(n+1)^4}=\frac{\pi^{4}}{90}.
\end{displaymath}
From Theorem 2.3 and Lemma 2.1 (c), we note that
\begin{equation}
\frac{1}{2!}\sum_{n=1}^{\infty}\frac{H_{n+1}^{(2)}+(H_{n+1})^{2}}{n(n+1)}=\sum_{n=1}^{\infty}\frac{1}{n}\sum_{k=0}^{n}\binom{n}{k}\frac{(-1)^{k}}{(k+1)^{3}}=3, \label{22}
\end{equation}
and that
\begin{equation}
\frac{1}{3!}\sum_{n=1}^{\infty}\frac{2H_{n+1}^{(3)}+3H_{n+1}H_{n+1}^{(2)}+(H_{n+1})^{3}}{n(n+1)}=\sum_{n=1}^{\infty}\frac{1}{n}\sum_{k=0}^{n}\binom{n}{k}(-1)^{k}\frac{1}{(k+1)^{4}}=4. \label{23} \\
\end{equation}
Therefore, by \eqref{22} and \eqref{23}, we obtain the following corollary.
\begin{corollary}
For $n\ge 1$, we have
\begin{equation*}
\sum_{n=1}^{\infty}\frac{H_{n+1}^{(2)}+(H_{n+1})^{2}}{n(n+1)}=3!,\quad
\sum_{n=1}^{\infty}\frac{2H_{n+1}^{(3)}+3H_{n+1}H_{n+1}^{(2)}+(H_{n+1})^{3}}{n(n+1)}=4!.
\end{equation*}
\end{corollary}
From \eqref{21}, Lemma 2.1 (b), \eqref{12} and \eqref{13}, we have
\begin{align}
F_{n}^{(4)}(x)&=\frac{d}{dx}\Big(2(-1)^{3}H_{n}(x,3)F_n(x)+3(-1)^{3}H_{n}(x,1)H_{n}(x,2)F_n(x)+(-1)^{3}(H_{n}(x,1))^{3}F_n(x)\Big)\label{24}    \\
&=6(-1)^{4}H_{n}(x,4)F_n(x)+2(-1)^{4}H_{n}(x,3)H_{n}(x,1)F_n(x)+3(-1)^{4}H_{n}(x,2)H_{n}(x,2)F_n(x)\nonumber \\
&+6(-1)^{4}H_{n}(x,1)H_{n}(x,3)F_n(x)+3(-1)^{4}\big(H_{n}(x,1)\big)^{2}H_{n}(x,2)F_n(x) \nonumber \\
&+3(-1)^{4}\big(H_{n}(x,1)\big)^{2}H_{n}(x,2)F_n(x)+(-1)^{4}\big(H_{n}(x,1)\big)^{4}F_n(x)\nonumber \\
&=(-1)^{4}\Big(6H_{n}(x,4)+8H_{n}(x,3)H_{n}(x,1)+3\big(H_{n}(x,2)\big)^{2}+6\big(H_{n}(x,1)\big)^{2}H_{n}(x,2)\nonumber\\
&+\big(H_{n}(x,1)\big)^{4}\Big)F_n(x)\nonumber\\
&=(-1)^{4}\frac{6H_{n}(x,4)+8H_{n}(x,3)H_{n}(x,1)+3(H_{n}(x,2))^{2}+6(H_{n}(x,1))^{2}H_{n}(x,2)+(H_{n}(x,1))^{4}}{\binom{x+n+1}{n}(x+1)},\nonumber
\end{align}
and
\begin{equation}
F_{n}^{(4)}(0)=(-1)^{4}\frac{6H_{n+1}^{(4)}+8H_{n+1}^{(3)}H_{n+1}+3(H_{n+1}^{(2)})^{2}+6(H_{n+1})^{2}H_{n+1}^{(2)}+(H_{n+1})^{4}}{n+1}. \label{25}
\end{equation}
On the other hand, by Lemma 2.1 (a), we get
\begin{equation}
F_{n}^{(4)}(x)=4!\sum_{k=0}^{n}\binom{n}{k}(-1)^{k-4}\frac{1}{(k+x+1)^{5}}= \int_{0}^{1}(1-t)^{n}(\log t)^{4}t^{x}dt,\label{26}
\end{equation}
and
\begin{equation}
F_{n}^{(4)}(0)=4!\sum_{k=0}^{n}\binom{n}{k}(-1)^{k-4}\frac{1}{(k+1)^{5}}= \int_{0}^{1}(1-t)^{n}(\log t)^{4}dt.\label{27}
\end{equation}
By \eqref{24}-\eqref{27}, we get
\begin{align}
&\sum_{k=0}^{n}\binom{n}{k}(-1)^{k}\frac{1}{(k+x+1)^{5}}=\frac{1}{4!} \int_{0}^{1}(1-t)^{n}(\log t)^{4}t^{x}dt \label{28}\\
&=\frac{6H_{n}(x,4)+8H_{n}(x,3)H_{n}(x,1)+3(H_{n}(x,2))^{2}+6(H_{n}(x,1))^{2}H_{n}(x,2)+(H_{n}(x,1))^{4}}{4!\binom{x+n+1}{n}(x+1)},\nonumber
\end{align}
and
\begin{align}
&\sum_{k=0}^{n}\binom{n}{k}(-1)^{k}\frac{1}{(k+1)^{5}}=\frac{1}{4!}\int_{0}^{1}(1-t)^{n}(\log t)^{4} dt \label{29}\\
&=\frac{6H_{n+1}^{(4)}+8H_{n+1}^{(3)}H_{n+1}+3(H_{n+1}^{(2)})^{2}+6(H_{n+1})^{2}H_{n+1}^{(2)}+(H_{n+1})^{4}}{4!(n+1)}.\nonumber
\end{align}
Therefore, from \eqref{28} and \eqref{29}, by binomial inversion we obtain the following theorem.
\begin{theorem}
For $n\ge 0$, we have
\begin{align*}
&\sum_{k=0}^{n}\binom{n}{k}(-1)^{k}\frac{6H_{k}(x,4)+8H_{k}(x,3)H_{k}(x,1)+3(H_{k}(x,2))^{2}+6(H_{k}(x,1))^{2}H_{k}(x,2)+(H_{k}(x,1))^{4}}{4!\binom{x+k+1}{k}(x+1)}\\
&=\frac{1}{(n+x+1)^{5}},
\end{align*}
and
\begin{equation*}
\frac{1}{4!}\sum_{k=0}^{n}\binom{n}{k}(-1)^{k}\frac{6H_{k+1}^{(4)}+8H_{k+1}^{(3)}H_{k+1}+3(H_{k+1}^{(2)})^{2}+6(H_{k+1})^{2}H_{k+1}^{(2)}+(H_{k+1})^{4}}{k+1}=\frac{1}{(n+1)^{5}} .
\end{equation*}
\end{theorem}
By Theorem 2.5 and \eqref{4}, we have
\begin{align*}
&\sum_{n=0}^{\infty}\sum_{k=0}^{n}\binom{n}{k}(-1)^{k}\frac{6H_{k}(x,4)+8H_{k}(x,3)H_{k}(x,1)+3(H_{k}(x,2))^{2}+6(H_{k}(x,1))^{2}H_{k}(x,2)+(H_{k}(x,1))^{4}}{4!\binom{x+k+1}{k}(x+1)}\\
&=\sum_{n=0}^{\infty}\frac{1}{(n+x+1)^{5}}=\zeta(x+1,5).
\end{align*}
For $x=0$, we also have
\begin{align*}
&\frac{1}{4!}\sum_{n=0}^{\infty}\sum_{k=0}^{n}\binom{n}{k}(-1)^{k}\frac{6H_{k+1}^{(4)}+8H_{k+1}^{(3)}H_{k+1}+3(H_{k+1}^{(2)})^{2}+6(H_{k+1})^{2}H_{k+1}^{(2)}+(H_{k+1})^{4}}{k+1}\\
&=\sum_{n=0}^{\infty}\frac{1}{(n+1)^5}=\zeta(5).
\end{align*}
From \eqref{29} and Lemma 2.1 (c), we note that
\begin{align*}
    &\frac{1}{4!}\sum_{n=1}^{\infty}\frac{6H_{n+1}^{(4)}+8H_{n+1}^{(3)}H_{n+1}+3(H_{n+1}^{(2)})^{2}+6(H_{n+1})^{2}H_{n+1}^{(2)}+(H_{n+1})^{4}}{n(n+1)}\\
    &=\sum_{n=1}^{\infty}\frac{1}{n}\sum_{k=0}^{n}\binom{n}{k}(-1)^{k}\frac{1}{(k+1)^{5}}=5. \\
\end{align*}
Hence, we have
\begin{align*}
    &\sum_{n=1}^{\infty}\frac{6H_{n+1}^{(4)}+8H_{n+1}^{(3)}H_{n+1}+3(H_{n+1}^{(2)})^{2}+6(H_{n+1})^{2}H_{n+1}^{(2)}+(H_{n+1})^{4}}{n(n+1)}=5!.
\end{align*}
From Lemma 2.1 (a), (b), and \eqref{12}, we note that
\begin{align*}
    &(r+1)!\sum_{k=0}^{n}\binom{n}{k}(-1)^{k-r}\frac{1}{(k+x+1)^{r+2}}=\frac{d^{r}}{dx^{r}}(H_{n}(x,1)F_n(x)) \\
    &=\sum_{l=0}^{r}\binom{r}{l}\bigg(\frac{d^{l}}{dx^{l}}H_{n}(x,1)\bigg)\bigg(\frac{d^{r-l}}{dx^{r-l}}F_n(x)\bigg) \\
    &=\sum_{l=0}^{r}\binom{r}{l}l!(-1)^{l}H_{n}(x,l+1)\int_{0}^{1}(1-t)^{n}\big(\log t\big)^{r-l}t^{x}dt,
\end{align*}
which gives us
\begin{align}
\sum_{k=0}^{n}&\binom{n}{k}(-1)^{k}\frac{1}{(k+x+1)^{r+2}}. \label{30}\\
&=\frac{(-1)^r}{(r+1)!}\sum_{l=0}^{r}\binom{r}{l}l!(-1)^{l}H_{n}(x,l+1)\int_{0}^{1}(1-t)^{n}\big(\log t\big)^{r-l}t^{x}dt. \nonumber
\end{align}
From \eqref{30}, by binomial inversion and summing over $n$,  we obtain
\begin{align}
\frac{(-1)^r}{(r+1)!}&\sum_{n=0}^{\infty}\sum_{k=0}^{n}\binom{n}{k}(-1)^{k}\sum_{l=0}^{r}\binom{r}{l}l!(-1)^{l}H_{k}(x,l+1)\int_{0}^{1}(1-t)^{k}\big(\log t\big)^{r-l}t^{x}dt \label{31}\\
&=\sum_{n=0}^{\infty}\frac{1}{(n+x+1)^{r+2}}=\zeta(x+1, r+2),\nonumber
\end{align}
which, for $x=0$, yields
\begin{align*}
\frac{(-1)^r}{(r+1)!}&\sum_{n=0}^{\infty}\sum_{k=0}^{n}\binom{n}{k}(-1)^{k}\sum_{l=0}^{r}\binom{r}{l}l!(-1)^{l}H_{k+1}^{(l+1)}\int_{0}^{1}(1-t)^{k}\big(\log t\big)^{r-l}dt \\
&=\sum_{n=0}^{\infty}\frac{1}{(n+1)^{r+2}}=\zeta(r+2).
\end{align*}
By \eqref{30} with $x=0$ and Lemma 2.1 (c), we have
\begin{align}
\frac{(-1)^r}{(r+1)!}&\sum_{l=0}^{r}\binom{r}{l}l!(-1)^{l}\sum_{n=1}^{\infty}\frac{1}{n}H_{n+1}^{(l+1)}\int_{0}^{1}(1-t)^{n}(\log t)^{r-l}dt \label{32}\\
&=\sum_{n=1}^{\infty}\frac{1}{n}\sum_{k=0}^{n}\binom{n}{k}(-1)^{k}\frac{1}{(k+1)^{r+2}}=r+2. \nonumber
\end{align}
Thus, from \eqref{30}-\eqref{32}, we obtain the following theorem.
\begin{theorem}
For $n, r \ge 0$, we have the following identities:
\begin{align*}
&\sum_{k=0}^{n}\binom{n}{k}(-1)^{k}\frac{1}{(k+x+1)^{r+2}}
=\frac{(-1)^r}{(r+1)!}\sum_{l=0}^{r}\binom{r}{l}l!(-1)^{l}H_{n}(x,l+1)\int_{0}^{1}(1-t)^{n}\big(\log t\big)^{r-l}t^{x}dt, \\
&\frac{(-1)^r}{(r+1)!}\sum_{n=0}^{\infty}\sum_{k=0}^{n}\binom{n}{k}(-1)^{k}\sum_{l=0}^{r}\binom{r}{l}l!(-1)^{l}H_{k}(x,l+1)\int_{0}^{1}(1-t)^{k}\big(\log t\big)^{r-l}t^{x}dt \\
&=\zeta(x+1, r+2),
\end{align*}
and
\begin{equation*}
\sum_{l=0}^{r}\binom{r}{l}l!(-1)^{l}\sum_{n=1}^{\infty}\frac{1}{n}H_{n+1}^{(l+1)}\int_{0}^{1}(1-t)^{n}(\log t)^{r-l}dt=(-1)^r(r+2)!.
\end{equation*}
\end{theorem}
\section{Conclusion}
In this paper, we used elementary methods in order to derive some identities involving generalized harmonic numbers and generalized harmonic functions from the beta functions $F_{n}(x)=B(x+1,n+1),\,\,(n=0,1,2,\dots)$. In more detail, for every $r=2,3,4,5$, we showed that
\begin{equation}
\sum_{k=0}^{n}\binom{n}{k}(-1)^{k}\frac{1}{(k+x+1)^{r}},\quad \zeta(x+1,r),\quad r! \label{33}
\end{equation}
are all expressed in terms of generalized harmonic functions and generalized harmonic numbers. The methods employed in this paper could be continued further for $r=6,7,8,\dots$, which requires to find an explicit expression for $F_{n}^{(r-1)}(x)$ in terms of the generalized harmonic functions. In Theorem 2.6, for any $r=2,3,4,\dots$, we found expressions for \eqref{33} without deriving the explicit expressions for
$F_{n}^{(r-l)}(x)=\int_{0}^{1}(1-t)^{n}(\log t)^{r-l}dt$ in terms of the generalized harmonic functions.

\end{document}